\documentclass[12pt,reqno]{amsart}
\usepackage[english]{babel}
\usepackage{amsmath,amsfonts,amssymb,amsthm,amscd,latexsym}
\usepackage[T1]{fontenc}
\usepackage[utf8]{inputenc}

\usepackage{tikz}
\usetikzlibrary{arrows}
\usepackage{color}
\usepackage{cite}
\usepackage{multirow}
\usepackage{cite}

\newtheorem{theorem}{Theorem}[section]
\newtheorem{lemma}[theorem]{Lemma}

\newtheorem{corollary}[theorem]{Corollary}
\theoremstyle{definition}
\newtheorem{definition}[theorem]{Definition}

\theoremstyle{remark}

\numberwithin{equation}{section}

\begin{document}

\setcounter{page}{1}

\title[2-local derivations on Witt algebras]{2-local derivations on Witt algebras}

\author{Yueqiang Zhao}
\address{School of Mathematical Sciences, Hebei Normal University, Shijiazhuang 050016, Hebei, China}
\email{yueqiangzhao@163.com}

\author{Yang Chen}
\address{Mathematics Postdoctoral Research Center,
Hebei Normal University, Shijiazhuang 050016, Hebei, China}
\email{chenyang1729@hotmail.com}

\author{Kaiming Zhao}
\address{Department of Mathematics, Wilfrid
Laurier University, Waterloo, ON, Canada N2L 3C5,  and School of Mathematical Sciences, Hebei Normal University, Shijiazhuang 050016, Hebei, China}
\email{kzhao@wlu.ca}



\date{}
\maketitle

\begin{abstract} In this paper, we prove that every 2-local derivation on   Witt algebras $W_n, W_n^+$ or $W_n^{++} $  is a derivation for all $n=1,2,\cdots,\infty$. As a consequence we obtain that every 2-local derivation on any centerless generalized Virasoro algebra of higher rank is a derivation.

\


\end{abstract}

\section{Introduction}

In 1997, \v{S}emrl \cite{Sem} introduced the notion of 2-local derivations on algebras. Namely, for an associative algebra $A$, a map $\Delta : A\to A$
(not necessarily linear) is called a \textit{2-local
derivation} if for every $x,y\in A$, there exists a
derivation $D_{x,y} : A\to A$ such that
$\Delta(x)= D_{x,y} (x)$ and $\Delta(y)= D_{x,y}(y)$. The concept of 2-local derivation is actually an important and interesting property for an algebra. Recently, several papers have been devoted to similar notions and corresponding problems for   Lie (super)algebras $L$. The main problem in this subject is to determine all 2-local derivations, and to see  whether 2-local derivations automatically become a
derivation  of $L$. In \cite{AyuKudRak} it is proved that every 2-local derivation on a finite-dimensional
semi-simple Lie algebra $L$ over complex numbers is a derivation and that
each finite-dimensional nilpotent Lie algebra with dimension
larger than two admits a 2-local derivation which is not a
derivation. In \cite{Yus} and \cite{AY} Ayupov and Yusupov studied 2-local derivations on univariate Witt algebras.
In the present paper we study 2-local derivations on
multivariate Witt algebras.

The Witt algebra $W_n$ of vector fields on an $n$-dimensional torus  is the derivation Lie algebra of the Laurent
polynomial algebra $A_n=\mathbb{C}[t_1^{\pm1},t_2^{\pm1},\cdots,
t_n^{\pm1}]$. Witt algebras were one of the four classes of Cartan type Lie algebras originally introduced in 1909 by Cartan \cite{C} when he studied
infinite dimensional simple Lie algebras. Over the last two decades, the  representation theory of Witt algebras was  extensively
studied by many mathematicians and physicists; see for example \cite{BF, BMZ, GLLZ}. Very recently,  Billig and Futorny
\cite{BF}   obtained the classification for all simple Harish-Chandra
$W_n$-modules.

The paper is organized as follows. In Section 2 we recall some known results  and establish some related properties concerning   Witt algebras and prove an auxiliary result for later use. In Section 3 we prove that every 2-local derivation on Witt algebras $W_n$ is a derivation. As a consequence we obtain that every 2-local derivation on any centerless generalized Virasoro algebra of higher rank is a derivation. In Section 4 we obtain similar results for $W_\infty$. Finally, in Section 5 we show that the above methods and conclusions are applicable for Witt algebras $W_n^+,\ W_n^{++},\ W_\infty^{+}$ and $W_\infty^{++}$.

Throughout this paper, we denote by $\mathbb{Z}$, $\mathbb{N}$, $\mathbb{Z}_+$ and $\mathbb{C}$ the sets of  all integers, positive integers, non-negative integers and complex numbers respectively. All algebras are over $\mathbb{C}$.
\medskip

\section{The Witt algebras}

\medskip

In this section we recall definitions, symbols and establish some auxiliary
results for later use in this paper.

A derivation on a Lie algebra $L$ is a linear map
$D: L\rightarrow L$ which satisfies the Leibniz
law
$$
D([x,y])=[D(x),y]+[x, D(y)],\
\forall x,y\in L.$$ The set of all derivations of
$L$   is a Lie algebra and
is denoted by $\text{Der}(L)$. For any $a\in L$,
the map
$$\text{ad}(a): L\to L, \ \text{ad}(a)x=[a,x],\ \forall  x\in L$$ is a derivation and derivations of this form are called
\textit{inner derivations}. The set of all inner derivations of
$L$, denoted by $\text{Inn}(L),$ is an ideal in
$\text{Der}(L).$ A derivation $D$ of $L$ is a \textit{locally inner derivation} if for any finite subset $S$ of $L$ there exist a finite-dimensional subspace $V$ of $L$ containing $S$ and $a\in L$ such that $D\big|_V= \text{ad}(a)\big|_V$.

For $n\in\mathbb{N}$, let $A_n= \mathbb{C}[t_1^{\pm1},t_2^{\pm1},\cdots,t_n^{\pm1}]$ be the Laurent polynomial algebra and $W_n= \text{Der}(A_n)$ be the Witt algebra of vector fields  on an $n$-dimensional torus. Thus $W_n$ has a natural structure of an $A_n$-module, which is free of rank $n$. We choose $d_1= t_1\frac{\partial}{\partial{t_1}}, \ldots, d_n= t_n\frac{\partial}{\partial{t_n}}$ as a basis of this $A_n$-module:
$$W_n=\bigoplus_{i=1}^n A_n d_i.$$
Denote  $t^\alpha=t_1^{\alpha_1}\cdots t_n^{\alpha_n}$ for $\alpha= (\alpha_1, \ldots, \alpha_n)\in \mathbb{Z}^n$ and let $\{\epsilon_1, \ldots, \epsilon_n\}$ be the standard basis of $\mathbb{Z}^n$. Then we can write the Lie bracket in $W_n$ as follows:
$$[t^\alpha d_i, t^\beta d_j]= \beta_i t^{\alpha+ \beta} d_j- \alpha_j t^{\alpha+ \beta} d_i,\ i, j= 1, \ldots, n;\ \alpha, \beta\in \mathbb{Z}^n.$$
The subspace $\mathfrak{h}$ spanned by $d_1, \ldots, d_n$ is the Cartan subalgebra in $W_n$. We may write any nonzero element in $W_n$ as $\sum_{\alpha\in S} t^\alpha d_\alpha$, where $S$ is the finite subset consisting of all  $\alpha\in  \mathbb{Z}^n$ with $d_\alpha\in \mathfrak{h}\setminus\{0\}$. Suppose that $d_\alpha= c_1 d_1+ \cdots+ c_n d_n\in \mathfrak{h}$ and $\beta\in \mathbb{Z}^n$, let $$(d_\alpha, \beta)=c_1 \beta_1+ \cdots+ c_n \beta_n.$$ Then we get the following formula
$$[t^\alpha d_\alpha, t^\beta d_\beta]= t^{\alpha+ \beta}((d_\alpha, \beta)d_\beta- (d_\beta, \alpha)d_\alpha).$$

Denote by $\mathbb{Z}^\infty$ (resp. $\mathbb{Z}_+^\infty$, $\mathbb{C}^\infty$) the set of all infinite sequences $\alpha= (\alpha_1, \alpha_2, \alpha_3, \cdots)$ where $\alpha_i\in \mathbb{Z}$ (resp. $\mathbb{Z}_+$, $\mathbb{C}$) and only a finite number of $\alpha_i$ can be nonzero. We can obtain similar definition and formula for the Witt algebra $W_\infty=\text{Der}(\mathbb{C}[t_1^{\pm1},t_2^{\pm1},\cdots]).$

\begin{definition}\label{def} \textit{We call a vector $\mu=(\mu_1, \mu_2, \cdots,\mu_n )\in \mathbb{C}^n$ generic if $$\mu\cdot \alpha=\alpha_1\mu_1+\alpha_2\mu_2+ \cdots+\alpha_n\mu_n\neq 0,\,\,\forall\alpha=(\alpha_1, \alpha_2, \cdots,\alpha_n )\in \mathbb{Z}^n \setminus \{0\}.$$ }
\end{definition}

Note that a vector $\mu=(\mu_1, \mu_2, \cdots,\mu_n )\in \mathbb{C}^n$ is generic if and only if $\mu_1, \mu_2, \cdots,\mu_n $ are linearly independent over the rational numbers $\mathbb{Q}$.
There are always generic vectors. For example, $\mu= (\zeta, \zeta^2, \ldots, \zeta^n)\in \mathbb{C}^n$ is just a generic vector, where $\zeta$ is a algebraic number of degree more than $n$ or a transcendental number in $\mathbb{C}$. Let $\mu$ be a generic vector in $\mathbb{C}^n$ and $d_\mu= \mu_1 d_1+ \cdots+ \mu_n d_n$.
Then we have the Lie subalgebra of $W_n$:
$$W_n(\mu)= A_n d_\mu,$$
which is called (centerless) generalized Virasoro algebra of rank $n$, see \cite{PZ}.

We proof the following key lemma at first.

\begin{lemma}\label{keylem} \textit{For any $k\in \mathbb{Z}\setminus \{0\}$,
the centralizer of $(t_1^k+ \cdots+ t_n^k)d_\mu$ in $W_n$ is \\ $\mathbb{C}(t_1^k+ \cdots+ t_n^k)d_\mu$}.
\end{lemma}

\begin{proof} We firstly deal with the case $k>0$. Let $x\in  W_n\setminus\{0\}$ such that $[x, (t_1^k+ \cdots+ t_n^k)d_\mu]= 0$.
We may write
 \begin{equation}\label{1}x= \sum_{\alpha\in S} t^\alpha d_\alpha\end{equation}
 where $S$ is the finite subset consisting of  all $\alpha\in  \mathbb{Z}^n$ with $d_\alpha\in \mathfrak{h}\setminus\{0\}$.
Take a term $t^\beta d_\beta$ of maximal degree with respect to some $t_i$ in the right-hand-side of (\ref{1}). Suppose $\beta\neq 0$ (if it exists).

{\bf Claim 1.} We must have $t^\beta d_\beta=c_i t_i^k d_\mu$ for some $c_i\in\mathbb{C}. $

The  term $[t^\beta d_\beta, t_i^k d_\mu]$ is  of maximal degree with respect to $t_i$ in $[x, (t_1^k+ \cdots+ t_n^k)d_\mu]$. We must have
$$[t^\beta d_\beta, t_i^k d_\mu]= t^{\beta+ k\epsilon_i}((d_\beta, k\epsilon_i)d_\mu-
(d_\mu, \beta)d_\beta)= 0.$$ Since $(d_\mu, \beta)\neq 0$, we have $d_\beta= c_i d_\mu$ for some
$0\neq c_i\in \mathbb{C}$ and furthermore
$$(d_\beta, k\epsilon_i)d_\mu- (d_\mu, \beta)d_\beta= (c_i d_\mu, k\epsilon_i)d_\mu- (d_\mu, \beta)c_i d_\mu= c_i(d_\mu, k\epsilon_i- \beta)d_\mu= 0.$$
So $\beta= k\epsilon_i$. Claim 1 follows.

Now $c_i t_i^k d_\mu$ is the only possible term of maximal degree with respect to $t_i$ in  (\ref{1}). If there is such a
term, we consider $c_1 t_1^k d_\mu$ without loss of generality. For $i\neq 1$, to delete the term $[c_1t_1^k d_\mu, t_i^k d_\mu]= c_1 k(\mu_i- \mu_1)t_1^k t_i^k d_\mu\neq 0$ in $[x, (t_1^k+ \cdots+ t_n^k)d_\mu]$, there must be the term $c_1 t_i^k d_\mu$
in  (\ref{1}). Thus $c_i= c_1$ for all $1\leq i\leq n$. Let $x'= x- c_1(t_1^k+ \cdots+ t_n^k)d_\mu$. Then $[x', (t_1^k+ \cdots+ t_n^k)d_\mu]= 0$.
Applying Claim 1 to $x'$, we see that the term $d$ of maximal degree with respect to any $t_i$ in $x'$ is in the Cartan subalgebra $\mathfrak{h}$.
Therefore
$$[d, t_i^k d_\mu]= (d, k\epsilon_i) t_i^k d_\mu= 0,\ \forall i= 1,\ldots, n.$$
We obtain that  $d=0$ and must have  $x'= 0$. For $k<0$, we may take terms of minimal degree instead. So for any $k\in \mathbb{Z}\setminus \{0\}$, the centralizer of
$(t_1^k+ \cdots+ t_n^k)d_\mu$ is only $\mathbb{C}(t_1^k+ \cdots+ t_n^k)d_\mu$.
\end{proof}

From Proposition 4.1 and Theorem 4.3 in \cite{DZ} we know that
 any derivation on $W_n$ (resp. $W_\infty$) is inner (resp. locally inner). Then for the Witt algebra $W=W_n$ or $W_\infty$ the above definition of the 2-local derivation can be reformulated as follows. A map
$\Delta$ on $W$  is a 2-local derivation on $W$ if
for any two elements $x, y \in W$ there exists an element
$a_{x, y} \in W$ such that
$$\Delta(x) = [a_{x,y} ,x],\,\,\,  \Delta(y) = [a_{x,y}, y].$$
In general,  a 2-local derivation  $\Delta$ on a Lie algebra $L$ is a derivation if
for any two elements $x, y \in L$ and $c\in\mathbb{C}$ we have
$$\Delta(x+cy) =\Delta(x) +c\Delta(y) ,\,\,\,  \Delta([x, y]) = [\Delta(x),y] +[x,\Delta(y) ].$$

\section{2-Local derivations on $W_n$}

Now we shall give the main result concerning 2-local derivations on $W_n$.

\begin{theorem}\label{thm31} \textit{Every 2-local derivation
on  the Witt algebra $W_n$ is a derivation.}
\end{theorem}
For the proof of this Theorem we need to set up several Lemmas as preparations. Let $\mu$ be a generic vector in $\mathbb{C}^n$ and $d_\mu= \mu_1 d_1+ \cdots+ \mu_n d_n$.

\begin{lemma}\label{lem32} \textit{Let $\Delta$ be a 2-local derivation
on $W_n$ such that $\Delta(d_\mu)=0$.
Then for any nonzero element $x= \sum_{\alpha\in S} t^\alpha d_\alpha\in W_n$, where $S$ is  the finite subset consisting of  all $\alpha\in  \mathbb{Z}^n$ with $d_\alpha\in \mathfrak{h}\setminus\{0\}$, we have $\Delta(x)= \Delta(\sum_{\alpha\in S} t^\alpha d_\alpha)\in \sum_{\alpha\in S} \mathbb{C}t^\alpha d_\alpha$}.
\end{lemma}

\begin{proof}
For $x$ and $d_\mu$, there exists an element $a= \sum_{\beta\in \mathbb{Z}^n} t^\beta d_\beta\in W_n$, such that
$$\Delta(d_\mu)= [a,d_\mu], \ \Delta(x)= \Delta(\sum_{\alpha\in S} t^\alpha d_\alpha)= [a,\sum_{\alpha\in S} t^\alpha d_\alpha].$$
Then
$$0= \Delta(d_\mu)= [a,d_\mu]= [\sum_{\beta\in \mathbb{Z}^n} t^\beta d_\beta, d_\mu]= -\sum_{\beta\in \mathbb{Z}^n} t^\beta(d_\mu, \beta)d_\beta.$$
This means that $a= d_0\in \mathfrak{h}$. Thus
$$\Delta(x)= \Delta(\sum_{\alpha\in S} t^\alpha d_\alpha)= [d_0, \sum_{\alpha\in S} t^\alpha d_\alpha]\in \sum_{\alpha\in S} \mathbb{C} t^\alpha d_\alpha.$$
\end{proof}

\begin{lemma}\label{lem33} \textit{Let $\Delta$ be a 2-local derivation
on $W_n$ such that
$$
\Delta(d_\mu)=\Delta((t_1+ \cdots+ t_n)d_\mu)=0.
$$
Then $\Delta((t_1^k+ \cdots+ t_n^k)d_\mu)=0$ for all $k\in \mathbb{Z}$}.
\end{lemma}

\begin{proof}
Suppose $k\neq 0,1$. For $(t_1+ \cdots+ t_n)d_\mu$ and $ (t_1^k+ \cdots+ t_n^k)d_\mu $, there exists an element $a\in W_n$ such that
$$0= \Delta((t_1+ \cdots+ t_n)d_\mu)= [a, (t_1+ \cdots+ t_n)d_\mu],$$
$$\Delta((t_1^k+ \cdots+ t_n^k)d_\mu)= [a,(t_1^k+ \cdots+ t_n^k)d_\mu].$$
Then $a= c(t_1+ \cdots+ t_n)d_\mu$ for some  $c\in \mathbb{C}$ by Lemma \ref{keylem}. Using Lemma \ref{lem32} we have
$$\Delta((t_1^k+ \cdots+ t_n^k)d_\mu)= [c(t_1+ \cdots+ t_n)d_\mu, (t_1^k+ \cdots+ t_n^k)d_\mu]\in \sum_{i=1}^n \mathbb{C} t_i^k d_\mu.$$
It implies that $c= 0$, otherwise there exists a nonzero term $c(k-1)\mu_1 t_1^{k+1}d_\mu$ in $\Delta((t_1^k+ \cdots+ t_n^k)d_\mu)$ but not in $\sum_{i=1}^n \mathbb{C} t_i^k d_\mu$. Hence $\Delta((t_1^k+ \cdots+ t_n^k)d_\mu)=0$ for all $k\in \mathbb{Z}$.
\end{proof}

\begin{lemma}\label{lem34}
 \textit{Let $\Delta$ be a 2-local derivation on $W_n$ such that $\Delta((t_1^k+ \cdots+ t_n^k)d_\mu)=0$ for all $k\in \mathbb{Z}$. Then $\Delta= 0$}.
\end{lemma}

\begin{proof}
Take an arbitrary nonzero element $x= \sum_{\alpha\in S} t^\alpha d_\alpha\in W_n$, where $S$ is the finite subset consisting of all $\alpha\in  \mathbb{Z}^n$ with $d_\alpha\in \mathfrak{h}\setminus\{0\}$. Let $n_x$ be an index such that $|\alpha_i|< n_x$ for any $\alpha\in S$. For the fixed $k> 2n_x$, there is an element $a\in W_n$ such that
$$0= \Delta((t_1^k+ \cdots+ t_n^k)d_\mu)= [a,(t_1^k+ \cdots+ t_n^k)d_\mu],$$
$$\Delta(\sum_{\alpha\in S} t^\alpha d_\alpha)= [a,\sum_{\alpha\in S} t^\alpha d_\alpha].$$
Then $a= c(t_1^k+ \cdots+ t_n^k)d_\mu$ for some $c\in \mathbb{C}$ by Lemma \ref{keylem}. Using Lemma \ref{lem32}, we have
$$\Delta(\sum_{\alpha\in S} t^\alpha d_\alpha)= [c(t_1^k+ \cdots+ t_n^k)d_\mu, \sum_{\alpha\in S} t^\alpha d_\alpha]\in \sum_{\alpha\in S} \mathbb{C}t^\alpha d_\alpha.$$
It implies that $c= 0$, otherwise there exists some $t_i$ whose degree is more than $n_x$ for any term in $\Delta(\sum_{\alpha\in S} t^\alpha d_\alpha)$. So $\Delta= 0$.
\end{proof}

Now we are in position to prove Theorem \ref{thm31}.

\textit{Proof of Theorem} \ref{thm31} Let $\Delta$ be a 2-local derivation on $W_n$. Take an element $a\in W_n$ such that
$$\Delta(d_\mu)= [a, d_\mu], \ \Delta((t_1+ \cdots+ t_n)d_\mu)= [a, (t_1+ \cdots+ t_n)d_\mu].$$
Set $\Delta_1=\Delta- \text{ad} (a).$ Then $\Delta_1$ is a 2-local
derivation such that
$$\Delta_1(d_\mu)=\Delta_1((t_1+ \cdots+ t_n)d_\mu)=0.$$
By Lemma \ref{lem33} and Lemma \ref{lem34}, it
follows that $\Delta_1=0.$ Thus $\Delta= \text{ad} (a)$ is a
derivation. The proof is complete.
\hfill$\Box$

By Theorem 3.4 in \cite{DZ1} any derivation on the generalized Virasoro algebra $W_n(\mu)$ can be seen as the restriction
of a inner derivation on $W_n$.  All the proofs in this section with minor modifications are valid
for the generalized Virasoro algebra $W_n(\mu)$. Therefore
we obtain the following consequence.

\begin{corollary}\label{lem36} \textit{
Let  $n\in\mathbb{N}$, and let $\mu\in\mathbb{C}^n$ be generic. Then any 2-local derivation
on the generalized Virasoro algebra $W_n(\mu)$  is a derivation.
}
\end{corollary}

\section{2-Local derivations on $W_\infty$}

In this section we shall determine all 2-local derivations on $W_\infty$.  For a given $n\in \mathbb{N}$, we know that $W_n=\text{Der}(\mathbb{C}[t_1^{\pm1},t_2^{\pm1},\cdots, t_n^{\pm1}])$ is a subalgebra of $W_\infty=\text{Der}(\mathbb{C}[t_1^{\pm1},t_2^{\pm1},\cdots])$. We still suppose that $\mu$ is a generic vector in $\mathbb{C}^n$ and $d_\mu= \mu_1 d_1+ \cdots+ \mu_n d_n$. The Cartan subalgebra of $W_\infty$ (resp. $W_n$) is denoted by $\mathfrak{h}_\infty$ (resp. $\mathfrak{h}_n$). For convenience we define
$$\aligned &K_n=\{\alpha\in \mathbb{Z}^\infty:\alpha_1=\alpha_2=\cdots=\alpha_n=0\},\\
&\mathfrak{h}_n'=\{h\in\mathfrak{h}_\infty:(h, \epsilon_i)=0\,\,\,\forall\,\,\, i=1,2,\cdots,n\}.\endaligned$$

\begin{lemma}\label{lem41} \textit{For a given $n\in \mathbb{N}$ and any $k\in \mathbb{Z}\setminus \{0\}$, the centralizer of $(t_1^k+ \cdots+ t_n^k)d_\mu$ in $W_\infty$ is
$$\sum_{ \beta\in K_n}\mathbb{C} t^\beta (t_1^k+ \cdots+ t_n^k)d_\mu+ \sum_{ \beta\in K_n} t^\beta \mathfrak{h}_n'.$$
}
\end{lemma}

\begin{proof} We firstly deal with the case $k>0$. Let $x\in  W_\infty\setminus\{0\}$ such that $[x, (t_1^k+ \cdots+ t_n^k)d_\mu]= 0$.
We may write
 \begin{equation}\label{2}x= \sum_{\alpha\in S} t^\alpha d_\alpha\end{equation}
 where $S$ is the finite subset consisting of  all $\alpha\in  \mathbb{Z}^\infty$ with $d_\alpha\in \mathfrak{h}_\infty\setminus\{0\}$.
Take a term $t^\gamma d_\gamma$ of maximal degree with respect to some $t_i$ ($1\leq i\leq n$) in the right-hand-side of (\ref{2}). Suppose that $\gamma\notin K_n$ (if it exists).

{\bf Claim 1.} We must have $t^\gamma d_\gamma=c_\beta^{(i)} t^{\beta+k\epsilon_i} d_\mu$ for some $\beta\in K_n$ and  $c_\beta^{(i)}\in\mathbb{C}.$

The term $[t^\gamma d_\gamma, t_i^k d_\mu]$ is of maximal degree with respect to $t_i$ in $[x, (t_1^k+ \cdots+ t_n^k)d_\mu]$. So
$$[t^\gamma d_\gamma, t_i^k d_\mu]= t^{\gamma+ k\epsilon_i}((d_\gamma, k\epsilon_i)d_\mu-
(d_\mu, \gamma)d_\gamma)= 0.$$ Since $(d_\mu, \gamma)\neq 0$, we deduce that  $d_\gamma= c_\beta^{(i)} d_\mu$,
$0\neq c_\beta^{(i)}\in \mathbb{C}$, yielding that
$$(d_\gamma, k\epsilon_i)d_\mu- (d_\mu, \gamma)d_\gamma= (c_\beta^{(i)} d_\mu, k\epsilon_i)d_\mu- (d_\mu, \gamma)c_\beta^{(i)} d_\mu= c_\beta^{(i)}(d_\mu, k\epsilon_i- \gamma)d_\mu= 0.$$
We see that  $\gamma= \beta+ k\epsilon_i$ for some $\beta\in K_n$. Claim 1 follows.

Now a finite sum $\sum_{ \beta\in K_n}c_\beta^{(i)} t^\beta t_i^k d_\mu$, where $c_\beta^{(i)}\in\mathbb{C}$, involves the only possible terms of maximal degree with respect to $t_i$ in (\ref{2}). If there is such a
finite sum, we consider $\sum_{ \beta\in K_n}c_\beta^{(1)} t^\beta t_1^k d_\mu$ without loss of generality. Hence, to delete the terms
$$[\sum_{ \beta\in K_n}c_\beta^{(1)} t^\beta t_1^k d_\mu, t_i^k d_\mu]= \sum_{ \beta\in K_n}c_\beta^{(1)} k(\mu_i- \mu_1)t^\beta t_1^k t_i^k d_\mu\neq 0,\ i\neq 1,$$
there must be the terms $\sum_{ \beta\in K_n}c_\beta^{(1)} t^\beta t_i^k d_\mu$ in (\ref{2}). Thus $c_\beta^{(i)}= c_\beta^{(1)}$ for all $1\leq i\leq n$.
Let
$$x'= x- \sum_{ \beta\in K_n}c_\beta^{(1)} t^\beta (t_1^k+ \cdots+ t_n^k)d_\mu.$$
Then $[x', (t_1^k+ \cdots+ t_n^k)d_\mu]= 0$. Applying Claim 1 to $x'$, we see that the terms of maximal degree with respect to any $t_i$ ($1\leq i\leq n$) in $x'$ are $\sum_{\beta\in K_n} t^\beta d_\beta$, $d_\beta\in \mathfrak{h}_\infty$. Therefore
$$[\sum_{ \beta\in K_n} t^\beta d_\beta,(t_1^k+ \cdots+ t_n^k)d_\mu]= \sum_{ \beta\in K_n, 1\le i\le n} t^{\beta+ k\epsilon_i} (d_\beta, k\epsilon_i) d_\mu= 0.$$
This implies $(d_\beta, \epsilon_i)= 0$ for all $ i= 1, \ldots, n$, i.e., $d_\beta\in\mathfrak{h}_n'$. Thus we must have
$x'\in \sum_{ \beta\in K_n} t^\beta \mathfrak{h}_n'$ by Claim 1.
For $k<0$, we may take terms of minimal degree instead. So for any $k\in \mathbb{Z}\setminus \{0\}$, the centralizer of $(t_1^k+ \cdots+ t_n^k)d_\mu$ in $W_\infty$ is only
$$\sum_{ \beta\in K_n}\mathbb{C} t^\beta (t_1^k+ \cdots+ t_n^k)d_\mu+ \sum_{ \beta\in K_n} t^\beta \mathfrak{h}_n'.$$
\end{proof}

\begin{lemma}\label{lem42} \textit{For a given $n\in \mathbb{N}$, let $\Delta$ be a 2-local derivation
on $W_\infty$ such that $\Delta(d_\mu)=0$.
Then for any nonzero element $x= \sum_{\alpha\in S} t^\alpha d_\alpha\in W_n$, where $S$ is the finite subset consisting of all $\alpha\in \mathbb{Z}^\infty$ with $\alpha_i= 0, i> n$ such that $d_\alpha\in \mathfrak{h}_n\setminus\{0\}$, we have
$$\Delta(x)= \Delta(\sum_{\alpha\in S} t^\alpha d_\alpha)\in \sum_{ \beta\in K_n}\sum_{\alpha\in S} \mathbb{C}t^{\alpha+ \beta} d_\alpha.$$
}
\end{lemma}

\begin{proof}
For $x$ and $d_\mu$, there exists an element $a= \sum_{\beta\in \mathbb{Z}^\infty} t^\beta d_\beta\in W_\infty$, such that
$$\Delta(d_\mu)= [a,d_\mu], \ \Delta(x)= \Delta(\sum_{\alpha\in S} t^\alpha d_\alpha)= [a,\sum_{\alpha\in S} t^\alpha d_\alpha].$$
Then
$$0= \Delta(d_\mu)= [a,d_\mu]= [\sum_{\beta\in \mathbb{Z}^\infty} t^\beta d_\beta, d_\mu]= -\sum_{\beta\in \mathbb{Z}^\infty} t^\beta(d_\mu, \beta)d_\beta.$$
This means that $a=\sum_{ \beta\in K_n} t^\beta d_\beta$. Thus
$$\Delta(x)= \Delta(\sum_{\alpha\in S} t^\alpha d_\alpha)= [\sum_{ \beta\in K_n} t^\beta d_\beta, \sum_{\alpha\in S} t^\alpha d_\alpha]\in \sum_{ \beta\in K_n}\sum_{\alpha\in S} \mathbb{C}t^{\alpha+ \beta} d_\alpha.$$
\end{proof}

\begin{lemma}\label{lem43} \textit{For a given $n\in \mathbb{N}$, let $\Delta$ be a 2-local derivation
on $W_\infty$ such that
$$
\Delta(d_\mu)=\Delta((t_1+ \cdots+ t_n)d_\mu)=0.
$$
Then $\Delta((t_1^k+ \cdots+ t_n^k)d_\mu)=0$ for all $k\in \mathbb{Z}$}.
\end{lemma}

\begin{proof}
Suppose $k\neq 0,1$. For $(t_1+ \cdots+ t_n)d_\mu$ and $ (t_1^k+ \cdots+ t_n^k)d_\mu$, there exists an element $a\in W_\infty$ such that
$$0= \Delta((t_1+ \cdots+ t_n)d_\mu)= [a, (t_1+ \cdots+ t_n)d_\mu],$$
$$\Delta((t_1^k+ \cdots+ t_n^k)d_\mu)= [a,(t_1^k+ \cdots+ t_n^k)d_\mu].$$
Then by Lemma \ref{lem41}
\begin{equation*}
a= \sum_{\beta\in K_n}c_\beta t^\beta (t_1+ \cdots+ t_n)d_\mu+ \sum_{ \beta\in K_n,\,\,  d_\beta\in \mathfrak{h}_n'}t^\beta d_\beta,\ c_\beta\in \mathbb{C}.
\end{equation*}
By Lemma \ref{lem42} we have
$$\Delta((t_1^k+ \cdots+ t_n^k)d_\mu)= [a, (t_1^k+ \cdots+ t_n^k)d_\mu]\in\sum_{ \beta\in K_n}\sum_{i=1}^n \mathbb{C} t^{k\epsilon_i+ \beta} d_\mu.$$
It implies that $c_\beta= 0$, otherwise there exists a nonzero term $c_\beta(k-1)\mu_1 t^{(k+1)\epsilon_1+ \beta} d_\mu$ in $\Delta((t_1^k+ \cdots+ t_n^k)d_\mu)$   not appearing  in $\sum_{ \beta\in K_n}\sum_{i=1}^n \mathbb{C} t^{k\epsilon_i+ \beta} d_\mu$. Hence $\Delta((t_1^k+ \cdots+ t_n^k)d_\mu)=0$ for all $k\in \mathbb{Z}$.
\end{proof}

\begin{lemma}\label{lem44}
 \textit{For a given $n\in \mathbb{N}$, let $\Delta$ be a 2-local derivation on $W_\infty$ such that $\Delta((t_1^k+ \cdots+ t_n^k)d_\mu)=0$ for all $k\in \mathbb{Z}$. Then $\Delta\big|_{W_n}= 0$}.
\end{lemma}

\begin{proof}
Take an arbitrary nonzero element $x= \sum_{\alpha\in S} t^\alpha d_\alpha\in W_n$, where $S$ is the finite subset consisting of all $\alpha\in \mathbb{Z}^\infty$ with $\alpha_i= 0, i> n$ such that $d_\alpha\in \mathfrak{h}_n\setminus\{0\}$. Let $n_x$ be an index such that $|\alpha_i|< n_x$ for any $\alpha\in S$. For the fixed $k> 2n_x$, there is an element $a\in W_\infty$ such that
$$0= \Delta((t_1^k+ \cdots+ t_n^k)d_\mu)= [a,(t_1^k+ \cdots+ t_n^k)d_\mu],$$
$$\Delta(\sum_{\alpha\in S} t^\alpha d_\alpha)= [a,\sum_{\alpha\in S} t^\alpha d_\alpha].$$
Then by Lemma \ref{lem41}
\begin{equation*}
a= \sum_{ \beta\in K_n}c_\beta t^\beta (t_1^k+ \cdots+ t_n^k)d_\mu+ \sum_{\beta\in K_n,\ d_\beta\in\mathfrak{h}_n'}t^\beta d_\beta,\ c_\beta\in \mathbb{C}.
\end{equation*}
By Lemma \ref{lem42}, we have
$$\Delta(\sum_{\alpha\in S} t^\alpha d_\alpha)= [a, \sum_{\alpha\in S} t^\alpha d_\alpha]\in \sum_{ \beta\in K_n}\sum_{\alpha\in S} \mathbb{C}t^{\alpha+ \beta} d_\alpha.$$
It implies that $c_\beta= 0$, otherwise there exists some $t_i$ ($1\leq i\leq n$) whose degree is more than $n_x$ for any term in $\Delta(\sum_{\alpha\in S} t^\alpha d_\alpha)$. So $\Delta\big|_{W_n}= 0$.
\end{proof}

\begin{theorem}\label{thm45} \textit{Any 2-local derivation
on  $W_\infty$ is a derivation.}
\end{theorem}

\begin{proof}
Let $\Delta$ be a 2-local derivation on $W_\infty$. For any $x,y\in W_\infty$, there exists  $n\in\mathbb{N}$ such that  $x,y\in W_n< W_\infty$. Take an element $a\in W_\infty$ such that
$$\Delta(d_\mu)= [a, d_\mu],\ \Delta((t_1+ \cdots+ t_n)d_\mu)= [a, (t_1+ \cdots+ t_n)d_\mu].$$
Set $\Delta_1=\Delta- \text{ad} (a).$ Then $\Delta_1$ is a 2-local
derivation such that
$$\Delta_1(d_\mu)=\Delta_1((t_1+ \cdots+ t_n)d_\mu)=0.$$
By Lemma \ref{lem43} and Lemma \ref{lem44}, it
follows that $\Delta_1\big|_{W_n}= 0$ and $\Delta\big|_{W_n}= \text{ad}(a)\big|_{W_n}$. Then
$$\Delta(c x+ y)= \text{ad}(a)(c x+ y)= c\text{ad}(a)(x)+ \text{ad}(a)(y)= c\Delta(x)+ \Delta(y),\,\,\forall\,\, c  \in \mathbb{C},$$
$$\Delta([x,y])= \text{ad}(a)([x,y])= [\text{ad}(a)(x), y]+ [x, \text{ad}(a)(y)]= [\Delta(x), y]+ [x, \Delta(y)].$$
 Therefore $\Delta$ is a derivation.
\end{proof}

\section{2-Local derivations on $W_n^+,\ W_n^{++},\ W_\infty^{+}$ and $W_\infty^{++}$}

 For  $n\in\mathbb{N}$, we have the Witt algebra $W_n^+=\text{Der}(\mathbb{C}[t_1 ,t_2 ,\cdots, t_n])$ which is a subalgebra of $W_n$. We use $ \mathfrak{h}_n$ to denote the Cartan subalgebra of $W_n$ which is also a  Cartan subalgebra (not unique) of $W_n^+$. We know that
$$W_n^+= \sum_{\alpha\in \mathbb{Z}_+^n} t^\alpha \mathfrak{h}_n+ \sum_{i=1}^n \mathbb{C} t_i^{-1} d_i.$$
Furthermore $W_n^+$ has a subalgebra
$$W_n^{++}= \sum_{\alpha\in \mathbb{Z}_+^n} t^\alpha \mathfrak{h}_n.$$
It is well-known that $W_n^+$ is a simple Lie algebra, but $W_n^{++}$ is not.

Similarly, the Witt algebra $W_\infty^+=\text{Der}(\mathbb{C}[t_1 ,t_2 ,\cdots, ])$   is a subalgebra of $W_\infty$. We use $ \mathfrak{h}_\infty$ to denote the Cartan subalgebra of $W_\infty$ which is also a  Cartan subalgebra (not unique) of $W_\infty^+$. We know that
$$W_\infty^+= \sum_{\alpha\in \mathbb{Z}_+^\infty} t^\alpha \mathfrak{h}_\infty+ \sum_{i=1}^\infty \mathbb{C} t_i^{-1} d_i.$$
Furthermore $W_\infty^+$ has a subalgebra
$$W_\infty^{++}= \sum_{\alpha\in \mathbb{Z}_+^\infty} t^\alpha \mathfrak{h}_\infty.$$
It is well-known that $W_\infty^+$ is a simple Lie algebra, but $W_\infty^{++}$ is not.

 From Proposition 4.1 and Theorem 4.3 in \cite{DZ} we know that any derivation on $W_n^+$ (resp. $W_\infty^+$) is inner (resp. locally inner).
 Using   same arguments as the proof of  Proposition 3.3 in \cite{DZ1} we can show that
any derivation on  $W_n^{++}$ (resp.  $W_\infty^{++}$) is inner (resp. locally inner).

Hence, the proofs and conclusions with slight modifications  in Section 3 (resp. Section 4) are applicable to $W_n^+$ and $W_n^{++}$ (resp. $W_\infty^+$ and $W_\infty^{++}$). It is routine to verify this. We omit the details and directly    state   the following theorem.

\begin{theorem}\label{thm51} \textit{Let $n=1,2,\cdots,\infty$. Then every 2-local derivation
on  the Witt algebra $W_n^+$ or $\ W_n^{++}$ is a derivation.}
\end{theorem}

\vspace{2mm}
\noindent
{\bf Acknowledgments. } This research is partially supported by NSFC (11871190) and NSERC (311907-2015).  The authors are grateful to the referee for some good suggestions and for  providing \cite{AKY} where
the results in our paper for $W_n$ and $W_n^+$ were also obtained using different methods.

\end{document}